\documentclass[a4paper,12pt,reqno]{amsart}

\usepackage{amsmath}
\usepackage{amscd}
\usepackage{amssymb}
\usepackage{mathrsfs}
\usepackage[left=2.5cm,right=2.5cm,bottom=3cm,top=3cm]{geometry}
\newtheorem{thm}{Theorem}[section]
\newtheorem{cor}[thm]{Corollary}
\newtheorem{conj}[thm]{Conjecture}
\newtheorem{exmp}[thm]{Example}
\newtheorem{lem}[thm]{Lemma}

\theoremstyle{definition}

\theoremstyle{remark}
\newtheorem{rem}[thm]{Remark}

\numberwithin{equation}{section}
\begin{document}

\title{Infinite-time singularity type of the K\"ahler-Ricci flow}

\author{Yashan Zhang}

\email{yashanzh@163.com}
\thanks{The author is partially supported by the Project MYRG2015-00235-FST of the University of Macau}

\begin{abstract}
For the K\"ahler-Ricci flow on a compact K\"ahler manifold with semi-ample canonical line bundle, we prove the singularity type at infinity does not depend on the choice of the initial metric. We also provide new simple proofs for some existing classification results on infinite-time singularity type of the K\"ahler-Ricci flow.
\end{abstract}

\maketitle

\section{Introduction}
We study the singularity type of long time solution of the K\"ahler-Ricci flow on compact K\"ahler manifolds. Thanks to the maximal existence time theorem of the K\"ahler-Ricci flow \cite{C,TZo,Ts}, the existence of long time solution is equivalent to nefness of the canonical line bundle. In this paper, motivated by the Abundance Conjecture (which predicts that if the canonical line bundle of an algebraic manifold is nef, then it is semi-ample), we will further restrict our discussions on $n$-dimensional K\"ahler manifold, denoted by $X$, with semi-ample canonical line bundle $K_X$, and so the Kodaira dimension $kod(X)\in\{0,1,\ldots,n\}$.
\par Given an arbitrary K\"ahler metric $\omega_0$ on $X$, let $\omega=\omega(t)_{t\in[0,\infty)}$ be the solution to the K\"ahler-Ricci flow
\begin{equation}\label{nkrf}
\partial_t\omega=-Ric(\omega)-\omega
\end{equation}
running from $\omega_0$. Recall from \cite{Ha93} that the infinite-time singularities of the K\"ahler-Ricci flow are divided into two types \textrm{IIb} and \textrm{III}. Precisely, we say a long time solution of the K\"ahler-Ricci flow \eqref{nkrf} is of type \textrm{IIb} if
$$\sup_{X\times[0,\infty)}|Rm(\omega(t))|_{\omega(t)}=\infty$$
and of type  \textrm{III} if
$$\sup_{X\times[0,\infty)}|Rm(\omega(t))|_{\omega(t)}<\infty.$$
There are nice results in the study of infinite-time singualry type of the K\"ahler-Ricci flow, see e.g. \cite{FZ,Gi,Ha,ToZy}. In particular, for the K\"ahler-Ricci flow on K\"ahler manifold with semi-ample canonical line bundle, Tosatti-Zhang \cite{ToZy} classified the infinite-time singularity type for many cases. More precisely, if we let
$$f:X\to X_{can}\subset \mathbb{CP}^N$$
be the semi-ample fibration with connected fibers induced by pluricanonical system of $X$, where $X_{can}$, a $kod(X)$-dimensional irreducible normal projective variety, is the canonical model of $X$, and $V\subset X_{can}$ be the singular set of $X_{can}$ together with the critical values of $f$, and $X_y=f^{-1}(y)$ be a smooth fiber for $y\in X_{can}\setminus V$, then \cite[Theorems 1.5, 1.6]{ToZy} reads

\begin{thm}\cite[Theorems 1.5, 1.6]{ToZy}\label{TZ}
Let $X$ be a compact K\"ahler manifold with semi-ample $K_X$ and consider a solution of the K\"ahler-Ricci flow \eqref{nkrf}.
\begin{itemize}
\item[(1)] Suppose $kod(X)=0$.
\begin{itemize}
\item If $X$ is a finite quotient of a torus, then the solution is of type \textrm{III}.
\item If $K_X$ is not a finite quotient of a torus, the the solution is of type \textrm{IIb}.
\end{itemize}
\item[(2)] Suppose $kod(X)=n$.
\begin{itemize}
\item If $K_X$ is ample, then the solution is of type \textrm{III}.
\item If $K_X$ is not ample, the the solution is of type \textrm{IIb}.
\end{itemize}
\item[(3)] Suppose $0<kod(X)<n$.
\begin{itemize}
\item If $X_y$ is not a finite quotient of a torus, then the solution is of type \textrm{IIb}.
\item If $X_y$ is a finite quotient of a torus and $V=\emptyset$, then the solution is of type \textrm{III}.
\end{itemize}
\item[(4)] Suppose $n=2$ and $kod(X)=1$, then the solution is of type \textrm{III} if and only if the only singular fibers on $f$ are of type $mI_0$, $m>1$.
\end{itemize}
\end{thm}
Note that if $n=dim(X)=1$, then the conclusions in items (1) and (2) of Theorem \ref{TZ} were already shown by Hamilton \cite{Ha}; the first part of item (2) of Theorem \ref{TZ} was contained in Cao \cite{C} and Tsuji \cite{Ts}; and in the second part of Theorem \ref{TZ}(3), if in particular $X_y$ is a torus and also $V=\emptyset$, then that conclusion was first proved by Fong-Zhang \cite[Section 5]{FZ} by assuming $X$ is projective and the initial K\"ahler class is rational, and the projectivity and rationality assumptions were removed by Hein-Tosatti in \cite{HeTo} (also see \cite{Gi,SW} for certain special case, i.e. $X$ is a product). Moreover, when $n=2$, since Abundance Conjecture holds for K\"ahler surfaces, we can just assume $K_X$ is nef and a complete classification of singularity type is given by combining items (1), (2) and (4) of Theorem \ref{TZ}.

As an immediate consequence of Theorem \ref{TZ}, one has
\begin{cor}\cite{ToZy}\label{TZ1}
In the cases covered by Theorem \ref{TZ}, the singularity type does not depend on the choice of the initial metric.
\end{cor}

The classification given by Theorem \ref{TZ} (and hence Corollary \ref{TZ1}) is almost complete (for K\"ahler manifold with semi-ample canonical line bundle) but leaves open the case when\\

\begin{itemize}
\item[($\bigstar$)]\emph{$0<kod(X)<n$ ($n\ge3$) and the general fiber is a finite quotient of a torus and $V\neq\emptyset$}.\\
\end{itemize}

\par In general, it is expected (see \cite[Section 1]{ToZy}) that the singularity type in case ($\bigstar$) is also independent of the initial metric. More generally, a conjecture raised by Tosatti in \cite[Conjecture 6.7]{To.kawa} predicts

\begin{conj}
\cite[Conjecture 6.7]{To.kawa}
For the K\"ahler-Ricci flow on any compact K\"ahler manifold with nef canonical line bundle, the singularity type at infinity does not depend on the choice of the initial metric.
\end{conj}

\par In this note, we shall partially conform the conjectures mentioned in the above last two paragraphs. Our main result can be stated as follows.

\begin{thm}\label{main0}
Let $X$ be a compact K\"ahler manifold with semi-ample $K_X$. Then the singularity type of the K\"ahler-Ricci flow \eqref{nkrf} on $X$ does not depend on the choice of the initial metric.
\end{thm}

Note that the Abundance Conjecture holds for any $3$-dimensional compact K\"ahler manifold (see \cite{CHP}). Therefore, Theorem \ref{main0} implies the following
\begin{cor}
Let $X$ be a $3$-dimensional compact K\"ahler manifold with nef $K_X$. Then the singularity type of the K\"ahler-Ricci flow \eqref{nkrf} on $X$ does not depend on the choice of the initial metric.
\end{cor}
As we have seen from Corollary \ref{TZ1}, in those cases covered by Theorem \ref{TZ}, our Theorem \ref{main0} is not new. In fact, Theorem \ref{main0} is only new in case $(\bigstar)$, for which a classification of singularity type is still open. However, we would like to point out that, while the original proof of Theorem \ref{TZ} (and hence Corollary \ref{TZ1}) in \cite{ToZy} is based on analysis on the blowup limits of the K\"ahler-Ricci flow and a nice observation which relates singularity type to an algebraic condition (i.e. existence of some special rational curves), our proof of Theorem \ref{main0} takes a totally different (and more elementary) approach and does not involve the classification results, e.g. Theorem \ref{TZ}, on singularity type, and hence also provides an alternative proof for Corollary \ref{TZ1}. In fact, our arguments in this note will only use the maximum principle, and so are purely analytic.\\

\par
Theorem \ref{main0} can help to classify the singularity type. For example, by applying Theorem \ref{main0}, we can give a simple proof for item (1) of Theorem \ref{TZ}, see Example \ref{ex1} in Section \ref{ex}. Moreover, by the argument for Theorem \ref{main0}, we will provide a simple proof for item (2) of Theorem \ref{TZ}. To be more precise, let's state the harder part as follows.

\begin{thm}[second part of Theorem \ref{TZ} (2)]\label{ex2.0}
On a compact K\"ahler manifold $X$ with semi-ample $K_X$ and $kod(X)=n$, if there exists a type \textrm{III} solution $\omega=\omega(t)$ to the K\"ahler-Ricci flow \eqref{nkrf}, then $K_X$ is ample.
\end{thm}
Theorem \ref{ex2.0} was first proved by Tosatti-Zhang \cite{ToZy}, whose argument involves the existence of some special rational curves guarenteed by Moishezon \cite{Mo} and Kawamata \cite{Ka}. Later, an analytic proof for Theorem \ref{ex2.0} was given in Guo \cite[Theorem 1.2]{Guo} by using Cheeger-Colding's theory on Ricci limit space. In this note, we shall provide a simple analytic proof for Theorem \ref{ex2.0}, which will only involve the maximum principle argument, see Example \ref{ex2} in Section \ref{ex} for more details.\\

\par In the remaining part of this note, we will recall some necessary properties of the K\"ahler-Ricci flow in Section \ref{apriori} and then give a proof of Theorem \ref{main0} in Section \ref{proof}. In Section \ref{ex} we will present some related examples, which in particular contain new simple proofs for items (1) and (2) of Theorem \ref{TZ}.

\section{Properties of the K\"ahler-Ricci flow}\label{apriori}
We collect some necessary properties of the K\"ahler-Ricci flow.
\par Let
$$f:X\to X_{can}\subset \mathbb{CP}^N$$
be the semi-ample fibration with connected fibers induced by pluricanonical system, where $X_{can}$, a $kod(X)$-dimensional irreducible normal projective variety, is the canonical model of $X$. Let $\chi$ be a multiple of Fubini-Study metric on $\mathbb{CP}^N$ such that $f^*\chi$ is a smooth semi-positive representative of $-2\pi c_1(X)$ and $\Omega$ a fixed smooth positive volume form on $X$ with $\sqrt{-1}\partial\bar\partial\log\Omega=f^*\chi$. Given an arbitrary K\"ahler metric $\omega_0$ on $X$. We set $\omega_t:=e^{-t}\omega_0+(1-e^{-t})f^*\chi$ and reduce the K\"ahler-Ricci flow to the following parabolic complex Monge-Amp\`{e}re equation of $\varphi=\varphi(t)$,
$$\partial_t\varphi=\log\frac{e^{(n-k)t}(\omega_t+\sqrt{-1}\partial\bar\partial\varphi)^n}{\Omega}-\varphi$$
with $\varphi(0)=0$, where $n=dim(X)$, $k=kod(X)$ and $\omega=\omega_t+\sqrt{-1}\partial\bar\partial\varphi$ is the solution to the K\"ahler-Ricci flow \eqref{nkrf} with $\omega(0)=\omega_0$.

\begin{lem}\cite{ST4,Zo}\label{ST}
There exists a uniform constant $C$ such that on $X\times[0,\infty)$,
\begin{equation}\label{bound}
|\varphi|+|\partial_t\varphi|\le C.
\end{equation}
\end{lem}
The \eqref{bound} is proved in \cite{Zo} when $kod(X)=n$ and in \cite[Section 2]{ST4} when $0\le kod(X)\le n$.
\par We will also need Shi's estimates \cite{Sh} (also see \cite[Theorem 2.14]{SW})
\begin{lem}\label{Shi}
Let $\omega$ be a solution to the K\"ahler-Ricci \eqref{nkrf} on $X$ with type \textrm{III} singularity, then there exists a uniform constant $C$ such that on $X\times[0,\infty)$,
$$|\nabla_{g(t)}Rm(g(t))|_{g(t)}\le C.$$
Here $g(t)$ is the Riemannian metric associated to $\omega(t)$.
\end{lem}

\section{Proof of theorem \ref{main0}}\label{proof}
We now give a proof of our main result Theorem \ref{main0}. To this end, we will prove the following key observation, which implies Theorem \ref{main0} immediately and may be interesting in itself.
\begin{thm}\label{main}
Let $X$ be a compact K\"ahler manifold with semi-ample $K_X$. Assume there exists a K\"ahler metric $\tilde\omega_0$ such that the K\"ahler-Ricci flow $\eqref{nkrf}$ running from $\tilde\omega_0$ develops type \textrm{III} singularity. Then the K\"ahler-Ricci flow \eqref{nkrf} running form any K\"ahler metric on $X$ develops type \textrm{III} singularity.
\end{thm}

To prove Theorem \ref{main}, the basic idea is to compare two solutions of the K\"ahler-Ricci flow by using the Schwarz Lemma arguments and then to bound the curvature by using the maximum principle.
\begin{proof}[Proof of theorem \ref{main}]
Assume we are given a long time solution $\tilde\omega=\tilde\omega(t)$ of the K\"ahler-Ricci flow \eqref{nkrf} running from a K\"ahler metric $\tilde\omega_0$ on $X$, which is of type \textrm{III} singulary, i.e. for some uniform constant $C\ge1$,
\begin{equation}\label{type}
\sup_{X\times[0,\infty)}|Rm(\tilde\omega(t))|_{\tilde\omega(t)}\le C.
\end{equation}
Define $\tilde\omega_t,\tilde\varphi$ in the same manner as in last section (see Lemma \ref{ST}).
\par Let $\omega_0$ be an arbitrary K\"ahler metric on $X$ and $\omega$ the solution of the K\"ahler-Ricci flow \eqref{nkrf} running from $\omega_0$. We now show that $\omega$ is of type \textrm{III} singulary.
\par In the following, we use $g_{i\bar j},R_{i\bar jk\bar l},R_{i\bar j},\nabla, \Gamma_{jk}^i$ etc. to denote the local components, curvature, connection, Christoffel symbols etc. of $\omega$ and $\tilde g_{i\bar j},\tilde R_{i\bar jk\bar l},\tilde R_{i\bar j},\tilde\nabla,\tilde \Gamma_{jk}^i$ etc. to denote those of $\tilde\omega$.
\par Firstly, by a direct computation we have
\begin{equation}\label{trace0.0}
\Delta_{\omega}tr_\omega\tilde\omega=g^{\bar bp}g^{\bar qa}\tilde g_{p\bar q}R_{a\bar b}-g^{\bar ji}g^{\bar qp}\tilde R_{i\bar jp\bar q}+g^{\bar ji}g^{\bar qp}\tilde g^{\bar ba}\nabla_{i}\tilde g_{p\bar b}\nabla_{\bar j}\tilde g_{a\bar q}.
\end{equation}
On the other hand,
\begin{align}\label{trace0.1}
\partial_ttr_\omega\tilde\omega&=\partial_t(g^{\bar ji}\tilde g_{i\bar j})\nonumber\\
&=-g^{\bar jp}g^{\bar qi}\tilde g_{i\bar j}(-R_{p\bar q}-g_{p\bar q})+g^{\bar ji}(-\tilde R_{i\bar j}-\tilde g_{i\bar j})\nonumber\\
&=g^{\bar jp}g^{\bar qi}\tilde g_{i\bar j}R_{p\bar q}-tr_\omega(Ric(\tilde\omega)).
\end{align}
By combining \eqref{trace0.0} and \eqref{trace0.1}, we get the evolution of $tr_{\omega}\tilde\omega$:
\begin{equation}
(\partial_t-\Delta_{\omega})tr_{\omega}\tilde\omega=-tr_{\omega}(Ric(\tilde\omega))+g^{\bar ji}g^{\bar qp}\tilde R_{i\bar jp\bar q}-g^{\bar ji}g^{\bar qp}\tilde g^{\bar ba}\nabla_{i}\tilde g_{p\bar b}\nabla_{\bar j}\tilde g_{a\bar q}\nonumber.
\end{equation}
By assumption, the curvature of $\tilde\omega$ is uniformly bounded (see \eqref{type}), so we can find some uniform constant $C\ge1$ such that
\begin{equation}\label{trace0}
(\partial_t-\Delta_{\omega})tr_{\omega}\tilde\omega\le Ctr_{\omega}\tilde\omega+C(tr_{\omega}\tilde\omega)^2-g^{\bar ji}g^{\bar qp}\tilde g^{\bar ba}\nabla_{i}\tilde g_{p\bar b}\nabla_{\bar j}\tilde g_{a\bar q}.
\end{equation}
Consequently, there exists two uniform positive constants $C_0$ and $C$ such that (see \cite{Y,C})
\begin{equation}\label{trace1}
(\partial_t-\Delta_{\omega})\log tr_{\omega}\tilde\omega\le C_0tr_{\omega}\tilde\omega+C.
\end{equation}
On the other hand, by Lemma \ref{ST} we have
$$(\partial_t-\Delta_{\omega})\tilde\varphi=\partial_t\tilde\varphi-tr_{\omega}\tilde\omega+tr_{\omega}\tilde\omega_t\le -tr_{\omega}\tilde\omega+tr_{\omega}\tilde\omega_t+C$$
and
$$(\partial_t-\Delta_{\omega})\varphi=\partial_t\varphi-n+tr_{\omega}\omega_t\ge tr_{\omega}\omega_t-C.$$
Now, for the constant $C_0$ in \eqref{trace1}, we fix a constant $A\ge C_0+1$ such that
$$(C_0+1)\tilde\omega_0\le A\omega_0$$
and so
$$(C_0+1)\tilde\omega_t\le A\omega_t.$$
Then we get
$$(\partial_t-\Delta_{\omega})((C_0+1)\tilde\varphi-A\varphi)\le-(C_0+1)tr_{\omega}\tilde\omega+C$$
and hence,
\begin{equation}\label{trace2}
(\partial_t-\Delta_{\omega})(\log tr_{\omega}\tilde\omega+(C_0+1)\tilde\varphi-A\varphi)\le-tr_{\omega}\tilde\omega+C
\end{equation}
By applying the maximum principle to \eqref{trace2} and Lemma \ref{ST} we have, for some uniform constant $C\ge1$,
$$tr_{\omega}\tilde\omega\le C$$
on $X\times[0,\infty)$. Note that Lemma \ref{ST} implies that one can find a constant $C\ge1$ such that
$$C^{-1}\tilde\omega^n\le\omega^n\le C\tilde\omega^n$$
on $X\times[0,\infty)$. So,
\begin{equation}\label{trace2.1}
tr_{\tilde\omega}\omega\le\frac{1}{(n-1)!}(tr_{\omega}\tilde\omega)^{n-1}\frac{\omega^n}{\tilde\omega^n}\le C.
\end{equation}
Consequently, we have a uniform constant $C\ge1$ such that
\begin{equation}\label{trace3}
C^{-1}\tilde\omega\le\omega\le C\tilde\omega.
\end{equation}
on $X\times[0,\infty)$.\\

\par Combining \eqref{trace3} and \eqref{type}, we can modify the well-known arguments for the K\"ahler-Ricci flow to bound the curvature of $\omega$ (see e.g. \cite{C,PSS,SW}). For completeness, we present some details here. Firstly, define a tensor $\Psi=(\Psi_{ij}^k)$ by $\Psi_{ij}^k:=\Gamma_{ij}^k-\tilde\Gamma^k_{ij}$ and $S=|\Psi|_\omega^2$, Then
$$S=g^{\bar ji}g^{\bar lk}g^{\bar qp}\tilde\nabla_ig_{k\bar q}\overline{\tilde\nabla_{j}g_{l\bar p}}.$$
By \eqref{trace3} we find that the last term in \eqref{trace0}
\begin{align}
g^{\bar ji}g^{\bar qp}\tilde g^{\bar ba}\nabla_{i}\tilde g_{p\bar b}\nabla_{\bar j}\tilde g_{a\bar q}&=g^{\bar ji}g^{\bar qp}\tilde g^{\bar ba}(\nabla_{i}-\tilde\nabla_i)\tilde g_{p\bar b}(\nabla_{\bar j}-\tilde\nabla_{\bar j})\tilde g_{a\bar q}\nonumber\\
&=g^{\bar ji}g^{\bar qp}\tilde g^{\bar ba}(-\Psi_{ip}^d)\tilde g_{d\bar b}(-\overline{\Psi_{jq}^e})\tilde g_{a\bar e}\nonumber\\
&\ge C^{-1}S,
\end{align}
Then we easily sees from \eqref{trace0} that, for some uniform constant $C\ge1$,
\begin{equation}\label{C30}
(\partial_t-\Delta_{\omega})tr_{\omega}\tilde\omega\le C-C^{-1}S.
\end{equation}
We need to compute the evolution of $S$. Firstly, recall \cite[(2.38)]{SW}:
\begin{align}\label{C31}
\Delta_{\omega}S=&2Re(g^{\bar ji}g^{\bar qp}g_{k\bar l}(\Delta_\omega\Psi_{ip}^{k})\overline{\Psi_{jq}^l})+|\nabla\Psi|_{\omega}^2+|\overline{\nabla}\Psi|_\omega^2\nonumber\\
&+R^{\bar ji}g^{\bar qp}g_{k\bar l}\Psi_{ip}^k\overline{\Psi_{jq}^l}+g^{\bar ji}R^{\bar qp}g_{k\bar l}\Psi_{ip}^k\overline{\Psi_{jq}^l}-g^{\bar ji}g^{\bar qp}R_{k\bar l}\Psi_{ip}^k\overline{\Psi_{jq}^l}.
\end{align}
Secondly, we have
\begin{equation}\label{C32}
\partial_t\Psi_{ip}^k=\partial_t\Gamma_{ip}^k-\partial_t\tilde\Gamma_{ip}^k=-\nabla_{i}R_{p}^{\ \,k}+\tilde\nabla_{i}\tilde R_{p}^{\ \,k}.
\end{equation}
Note that in \eqref{C32} both $\partial_t\Gamma_{ip}^k$ and $\partial_t\tilde\Gamma_{ip}^k$ are tensors, we can compute them in normal coordinates with respect to $\omega$ and $\tilde\omega$, respectively.
\par Recall \cite[(2.43),(2.44)]{SW}:
\begin{equation}\label{C33}
\nabla_{\bar b}\Psi_{ip}^k=\tilde R_{i\bar bp}^{\ \ \ \,k}-R_{i\bar bp}^{\ \ \ \,k},
\end{equation}
and
\begin{equation}\label{C34}
\Delta_\omega\Psi_{ip}^k=\nabla^{\bar b}\tilde R_{i\bar bp}^{\ \ \ \,k}-\nabla_iR_{p}^{\ \,k}.
\end{equation}
Combining \eqref{C32} and \eqref{C34} gives
\begin{equation}\label{C35}
\partial_t\Psi_{ip}^k=\Delta_\omega\Psi_{ip}^k+\tilde\nabla_{i}\tilde R_{p}^{\ \,k}-\nabla^{\bar b}\tilde R_{i\bar bp}^{\ \ \ \,k}.
\end{equation}
and hence
\begin{align}\label{C36}
\partial_tS&=\partial_t|\Psi|_\omega^2\nonumber\\
&=2Re(g^{\bar ji}g^{\bar qp}g_{k\bar l}(\Delta_\omega\Psi_{ip}^k+\tilde\nabla_{i}\tilde R_{p}^{\ \,k}-\nabla^{\bar b}\tilde R_{i\bar bp}^{\ \ \ \,k})\overline{\Psi_{jq}^l})+S\nonumber\\
&+R^{\bar ji}g^{\bar qp}g_{k\bar l}\Psi_{ip}^k\overline{\Psi_{jq}^l}+g^{\bar ji}R^{\bar qp}g_{k\bar l}\Psi_{ip}^k\overline{\Psi_{jq}^l}-g^{\bar ji}g^{\bar qp}R_{k\bar l}\Psi_{ip}^k\overline{\Psi_{jq}^l}.
\end{align}
It follows from \eqref{C31} and \eqref{C36} that
\begin{equation}\label{C36.1}
(\partial_t-\Delta_\omega)S=S-|\nabla\Psi|_{\omega}^2-|\overline{\nabla}\Psi|_\omega^2+2Re(g^{\bar ji}g^{\bar qp}g_{k\bar l}(\tilde\nabla_{i}\tilde R_{p}^{\ \,k}-\nabla^{\bar b}\tilde R_{i\bar bp}^{\ \ \ \,k})\overline{\Psi_{jq}^l})
\end{equation}
Note that, by Lemma \ref{Shi} and \eqref{trace3}, $\tilde\nabla_{i}\tilde R_{p}^{\ \,k}$ is uniformly bounded with respect to $\omega$. Moreover,
\begin{align}\label{C37}
\nabla^{\bar b}\tilde R_{i\bar bp}^{\ \ \ \,k}&=g^{\bar ba}\nabla_a\tilde R_{i\bar bp}^{\ \ \ \,k}\nonumber\\
&=g^{\bar ba}(\nabla_a-\tilde\nabla_a)\tilde R_{i\bar bp}^{\ \ \ \,k}+g^{\bar ba}\tilde\nabla_a\tilde R_{i\bar bp}^{\ \ \ \,k}\nonumber\\
&=\Psi*Rm(\tilde\omega)+g^{\bar ba}\tilde\nabla_a\tilde R_{i\bar bp}^{\ \ \ \,k},
\end{align}
where $\Psi*Rm(\tilde\omega)$ means some linear combination of products of $\Psi$ and $Rm(\tilde\omega)$ by contraction using $\omega$. Then, combining Lemma \ref{Shi},  \eqref{trace3}, Cauchy inequality and above facts, we conclude from \eqref{C36.1} that, for some uniform constant $C\ge1$,
\begin{equation}\label{C38}
(\partial_t-\Delta_\omega)S\le CS+C-|\nabla\Psi|_{\omega}^2-|\overline{\nabla}\Psi|_\omega^2.
\end{equation}

Using \eqref{C30} and \eqref{C38}, we find a sufficiently large constant $A$ such that, for some constant $C\ge1$,
\begin{equation}\label{C39}
(\partial_t-\Delta_\omega)(S+Atr_{\omega}\tilde\omega)\le-S+C.
\end{equation}
By applying the maximum principle to \eqref{C39}, we find a constant $C\ge1$ such that
\begin{equation}\label{C310}
S\le C.
\end{equation}
Then, by noting from \eqref{C33}, \eqref{type} and \eqref{trace3} that
\begin{equation}\label{C311}
|\bar\nabla\Psi|_\omega^2=|\tilde R_{i\bar bp}^{\ \ \ \,k}-R_{i\bar bp}^{\ \ \ \,k}|_\omega^2\ge\frac{1}{2}|Rm(\omega)|_\omega^2-C,
\end{equation}
we know that \eqref{C38} implies
\begin{equation}\label{C312}
(\partial_t-\Delta_\omega)S\le C-\frac{1}{2}|Rm(\omega)|_\omega^2.
\end{equation}
Recall \cite[(2.61)]{SW}, for points where $|Rm(\omega)|_\omega>0$,
\begin{equation}\label{C313}
(\partial_t-\Delta_\omega)|Rm(\omega)|_\omega\le C|Rm(\omega)|_\omega^2-\frac{1}{2}|Rm(\omega)|_\omega.
\end{equation}
Therefore, for sufficiently large constant $A$, we have
$$(\partial_t-\Delta_\omega)(|Rm(\omega)|_\omega+AS)\le-\frac{1}{2}|Rm(\omega)|_\omega+C$$
and hence, by maximum principle, we have a constant $C\ge1$ such that
$$\sup_{X\times[0,\infty)}|Rm(\omega)|_\omega\le C.$$
\par Theorem \ref{main} is proved.
\end{proof}

\section{Examples}\label{ex}
Given Theorem \ref{main0}, we can check the singularity type if the K\"ahler-Ricci flow admits some special solutions. We show several examples as follows.
\begin{exmp}\label{ex1}
\emph{(An alternative proof for item (1) of Theorem \ref{TZ})} Let $X$ be a Calabi-Yau manifold. Thanks to Yau \cite{Y}, we can fix a Ricci-flat metric $\omega_{CY}$ on $X$. Then it is easy to see that
$\omega(t)=e^{-t}\omega_{CY}$ is the solution to the K\"ahler-Ricci flow \eqref{nkrf} with initial metric $\omega(0)=\omega_{CY}$. Then
$$|Rm(\omega(t))|_{\omega(t)}=e^{t}|Rm(\omega_{CY})|_{\omega_{CY}}$$
and hence, $\omega(t)$ develops type \textrm{III} singularity, i.e.
$$\sup_{X\times[0,\infty)}|Rm(\omega(t))|_{\omega(t)}<\infty,$$
if and only if $\omega_{CY}$ is flat. Therefore, by Theorem \ref{main0}, the K\"ahler-Ricci flow running from any K\"ahler metric on $X$ develops type \textrm{III} singularity if and only if there exists a flat metric on $X$, and so if and only if $X$ is a finite quotient of a torus. This is exactly item (1) of Theorem \ref{TZ}.
\end{exmp}
Next, we reprove item (2) of Theorem \ref{TZ} by using the arguments in this note. To this end, we first give a general lemma.

\begin{lem}\label{lem4.1}
Let $X$ be a compact K\"ahler manifold with semi-ample $K_X$. Assume there exists a K\"ahler metric $\hat\omega$ such that the K\"ahler-Ricci flow $\eqref{nkrf}$, $\omega=\omega(t)$, running from $\hat\omega$ develops type \textrm{III} singularity. Then there exists a uniform positive constant $C$ such that, on $X\times[0,\infty)$, we have
$$\omega\le C\hat\omega.$$
\end{lem}

\begin{proof}
The proof is similar to arguments in Section \ref{proof}. First, we have
\begin{equation}
\Delta_{\hat\omega}tr_{\hat\omega}\omega=\hat g^{\bar bp}\hat g^{\bar qa}g_{p\bar q}\hat R_{a\bar b}-\hat g^{\bar ji}\hat g^{\bar qp}R_{i\bar jp\bar q}+\hat g^{\bar ji}\hat g^{\bar qp}g^{\bar ba}\hat\nabla_ig_{p\bar b}\hat\nabla_{\bar j}g_{a\bar q}\nonumber
\end{equation}
and
$$\partial_ttr_{\hat\omega}\omega=-tr_{\hat\omega}(Ric(\omega))-tr_{\hat\omega}\omega.$$
Combining the above two equations gives
\begin{align}
(\partial_t-\Delta_{\hat\omega})tr_{\hat\omega}\omega=&-tr_{\hat\omega}(Ric(\omega))-tr_{\hat\omega}\omega-\hat g^{\bar bp}\hat g^{\bar qa}g_{p\bar q}\hat R_{a\bar b}\nonumber\\
&+\hat g^{\bar ji}\hat g^{\bar qp}R_{i\bar jp\bar q}-\hat g^{\bar ji}\hat g^{\bar qp}g^{\bar ba}\hat\nabla_ig_{p\bar b}\hat\nabla_{\bar j}g_{a\bar q}\nonumber.
\end{align}
Since $\hat\omega$ is a fixed metric and $\omega$ is of type \textrm{III} singularity, one easily finds a uniform constant $C\ge1$ such that
\begin{equation}
(\partial_t-\Delta_{\hat\omega})tr_{\hat\omega}\omega\le Ctr_{\hat\omega}\omega+C(tr_{\hat\omega}\omega)^2-\hat g^{\bar ji}\hat g^{\bar qp}g^{\bar ba}\hat\nabla_ig_{p\bar b}\hat\nabla_{\bar j}g_{a\bar q}\nonumber
\end{equation}
and hence,
\begin{equation}\label{4.1}
(\partial_t-\Delta_{\hat\omega})\log tr_{\hat\omega}\omega\le Ctr_{\hat\omega}\omega+C.
\end{equation}
On the other hand, if we define $\hat\omega_t=e^{-t}\hat\omega+(1-e^{-t})f^*\chi$, $\hat\varphi=\hat\varphi(t)$ as in Section \ref{apriori}, then we have, for some uniform constant $C\ge1$,
$$(\partial_t-\Delta_{\hat\omega})\hat\varphi=\partial_t\hat\varphi-tr_{\hat\omega}\omega+tr_{\hat\omega}\hat\omega_t\le-tr_{\hat\omega}\omega+C,$$
where we have used Lemma \ref{bound} and the fact that, for some uniform constant $C\ge1$, $\hat\omega_t\le C\hat\omega$ on $X\times[0,\infty)$. Then, as in Section \ref{proof}, we can choose a sufficiently large constant $A$ such that
$$(\partial_t-\Delta_{\hat\omega})(\log tr_{\hat\omega}\omega+A\hat\varphi)\le -tr_{\hat\omega}\omega+C.$$
Now, by the maximum principle we see that
$$tr_{\hat\omega}\omega\le C.$$
Lemma \ref{lem4.1} is proved.
\end{proof}

\begin{rem}
An immediate consequence of Lemma \ref{lem4.1} is that, under the assumption of Lemma \ref{lem4.1}, the diameter along the flow is uniformly bounded from above. Then, since type \textrm{III} singularity implies uniform bound for Ricci curvature, we can conclude some results on the Gromov-Hausdorff convergence of the flow. For example, by using arguments in \cite{GTZ13} (also see \cite[Section 5]{Zy}) we can conclude that every Gromov-Hausdorff limit $(M,d_M)$ along the flow (i.e. there exists a time sequence $t_i\to\infty$ with $(X,\omega(t_i))\to(M,d_M)$ in Gromov-Hausdorff topology) contains an open dense subset homeomorphic to $X_{can}\setminus V$. Here $X_{can}$ and $V$  are the same as in Section 1.
\end{rem}
We are ready to give the second example.

\begin{exmp}\label{ex2}
\emph{(An alternative proof of item (2) of Theorem \ref{TZ})} For the first part, assume $X$ be a compact K\"ahler manifold with ample $K_X$. Thanks to Aubin \cite{A} and Yau \cite{Y}, there exists a unique K\"ahler-Einstein metric $\omega_{KE}$ satisfying
$$Ric(\omega_{KE})=-\omega_{KE}.$$
Then it is easy to see that $\omega(t)\equiv\omega_{KE}$ is the solution to the K\"ahler-Ricci flow \eqref{nkrf} with initial metric $\omega_{KE}$. Obviously, this solution develops type \textrm{III} singularity. Then, by Theorem \ref{main0}, the K\"ahler-Ricci flow running from any K\"ahler metric on $X$ develops type \textrm{III} singularity. This is exactly the first part of item (2) of Theorem \ref{TZ}, which was first proved by Cao \cite{C} and Tsuji \cite{Ts}.
\par Now, let's look at the second part. We need to show that, on compact K\"ahler manifold $X$ with semi-ample $K_X$ and $kod(X)=n$, if there exists a solution $\omega=\omega(t)$ to the K\"ahler-Ricci flow \eqref{nkrf} with type \textrm{III} singularity, then $K_X$ is ample. To see this, we first apply Lemma \ref{lem4.1} to see that
$$\omega(t)\le C\hat\omega$$
for some fixed K\"ahler metric $\hat\omega$ on $X$ and constant $C\ge1$. On the other hand, since $kod(X)=n$, Lemma \ref{bound} gives
$$\omega(t)^n\ge C^{-1}\hat\omega^n,$$
and so, as in \eqref{trace2.1} we have
$$tr_{\omega(t)}\hat\omega\le C$$
for some uniform constant $C\ge1$. In conclusion, we have proved that, for some uniform constant $C\ge1$, one has
\begin{equation}\label{4.2}
C^{-1}\hat\omega\le\omega\le C\hat\omega.
\end{equation}
Having \eqref{4.2}, one can use the well-known arguments in the K\"ahler-Ricci flow (see e.g. \cite{C,PSS,SW}) to obtain higher order estimates of $\omega(t)$ and then, by the Arzela-Ascoli theorem we can choose a time sequence $t_i\to\infty$ such that $\omega(t_i)\to\omega_\infty$ in $C^\infty(X,\hat\omega)$-topology, where $\omega_\infty$ is a smooth K\"ahler metric on $X$ and $\omega_\infty\in -2\pi c_1(X)$. Therefore, having a smooth positive representative, $K_X$ is ample by Kodaira embedding theorem. The proof is completed.
\end{exmp}

The last example is a special case of item (3) of Theorem \ref{TZ}.
\begin{exmp}
Let $Y$ be a Calabi-Yau manifold and $Z$ a compact K\"ahler manifold with ample $K_Y$. Set $X=Y\times Z$. It was first proved in \cite{Gi,SW} that the K\"aher-Ricci flow running from any K\"ahler metric on $X$ develops type \textrm{III} singularity if $Y$ is a torus. Note that this result is a special case of the second part of item (3) of Theorem \ref{TZ}. Moreover, according to item (3) of Theorem \ref{TZ}, the K\"ahler-Ricci flow on $X$ develops type \textrm{III} singularity if and only if $Y$ is a finite quotient of a torus. We can apply our Theorem \ref{main0} to reprove this result. Indeed, as before, we only need to check some special solution. Precisely, if we let $\omega_X=\omega_Y+\omega_Z$, where $\omega_Y$ is a Calabi-Yau metric on $Y$ and $\omega_Z$ is the K\"ahler-Einstein metric on $Z$, then the solution of the K\"ahler-Ricci flow \eqref{nkrf} running from $\omega_X$ is given by
$$\omega(t)=e^{-t}\omega_Y+\omega_Z.$$
Then we have
$$|Rm(\omega(t))|^2_{\omega(t)}=e^{2t}|Rm(\omega_Y)|^2_{\omega_Y}+|Rm(\omega_Z)|^2_{\omega_Z},$$
from which we see that $\omega(t)$ develops type \textrm{III} singularity if and only if $\omega_Y$ is flat. Therefore, by Theorem \ref{main0}, the K\"ahler-Ricci flow \eqref{nkrf} running from any K\"ahler metric on $X$ develops type \textrm{III} singularity if and only if there exists a flat metric on $Y$, and so if and only if $Y$ is a finite quotient of a torus.
\end{exmp}

\section*{Acknowledgements}
The author is grateful to Professor Huai-Dong Cao for constant encouragement and support. He also thanks Professors Huai-Dong Cao and Valentino Tosatti for their interest in this work and some useful comments on a previous version of this paper and the referee for giving a useful comment.

\end{document}